\documentclass[12pt,oneside,reqno]{amsart}
\usepackage[margin=100pt]{geometry}

\usepackage{graphicx} 

\usepackage{amsmath, amsthm, amssymb, amsfonts, mathtools}
\usepackage{mathrsfs}
\usepackage{hyperref}
\usepackage{subcaption}

\usepackage[colorinlistoftodos]{todonotes}

\newcommand{\norm}[1]{\lVert#1\rVert}

\newtheorem{theorem}{Theorem}
\newtheorem*{theorem*}{Theorem}
\newtheorem{conjecture}{Conjecture}
\newtheorem{lemma}{Lemma}

\newtheorem{corollary}{Corollary}

\newcommand{\sfE}{\mathsf{E}}
\newcommand{\C}{\mathbb{C}}
\newcommand{\D}{\mathbb{D}}
\newcommand{\N}{\mathbb{N}}
\newcommand{\R}{\mathbb{R}}
\newcommand{\T}{\mathbb{T}}

\title{Chebyshev polynomials corresponding to a vanishing weight}
\author{Alex Bergman$^{1}$ \& Olof Rubin$^{2}$}
\address{Centre for Mathematical Sciences, Lund University\\Box 118, SE-22100, Lund,  Sweden\\}
\thanks{$^1$ Centre for Mathematical Sciences, Lund University, Box 118, 22100 Lund, Sweden.
E-mail: alex.bergman@math.lu.se}
\thanks{$^2$ Centre for Mathematical Sciences, Lund University, Box 118, 22100 Lund, Sweden.
E-mail: olof.rubin@math.lth.se}
\date{\today}

\begin{document}

\maketitle

\begin{abstract}
We consider weighted Chebyshev polynomials on the unit circle corresponding to a weight of the form $(z-1)^s$ where $s>0$. For integer values of $s$ this corresponds to prescribing a zero of the polynomial on the boundary.
As such, we extend findings of Lachance, Saff and Varga \cite{lsv79}, to non-integer $s$. Using this generalisation, we are able to relate Chebyshev polynomials on lemniscates and other, more established, categories of Chebyshev polynomials. An essential part of our proof involves the broadening of the Erd\H{o}s--Lax inequality to encompass powers of polynomials. We believe that this particular result holds significance in its own right. 
\end{abstract}
\medskip	
	
	
\medskip	
	
\section{Introduction}
Given a compact set $\sfE\subset \C$ and a continuous weight function $w:\sfE\rightarrow [0,\infty)$, a classical problem within the field of approximation theory concerns finding parameters $a_1^\ast,\dotsc,a_n^\ast\in \C$ satisfying
\begin{equation}
    \max_{z\in \sfE}\left|w(z)\prod_{k=1}^{n}(z-a_k^\ast) \right|= \min_{a_1,\dotsc,a_n\in \C}\max_{z\in \sfE}\left|w(z)\prod_{k=1}^{n}(z-a_k)\right|.
    \label{eq:chebyshev_equation}
\end{equation}
The existence of a minimiser is guaranteed through a compactness argument. Further, there is a unique minimiser if $w$ is non-zero at infinitely many points of $\sfE$, implicitly imposing that $\sfE$ contains infinitely many points. We denote this minimiser with
\[T_n^w(z) = \prod_{k=1}^{n}(z-a_k^\ast).\]

By definition, $T_n^w$ is the unique monic polynomial of degree $n$ whose deviation from $0$ when multiplied with $w$ is as small as possible. It is called the weighted Chebyshev polynomial corresponding to $w$. These polynomials exhibit further structure when $\sfE$ is a real set. In this case, $T_n^{w}$ is a real polynomial characterised by possessing an alternating set consisting of $n+1$ points. By this we mean that if $P$ is a monic polynomial of degree $n$, corresponding to which there are associated points $x_0<x_1<\cdots <x_{n}$ in $\sfE$ such that
\begin{equation}
    w(x_j)P(x_j) = (-1)^{n-j}\max_{z\in \sfE}|w(z)P(z)|,
    \label{eq:alternation_theorem}
\end{equation}
then $P = T_n^{w}$, see e.g. \cite[Theorem 7]{nsz21}. Such a characterisation does not, however, hold for arbitrary compact subsets of the complex plane. For further results concerning the basic theory of weighted Chebyshev polynomials and proofs regarding existence and uniqueness, we refer the reader to \cite{achieser56, csz-I,csz-IV, lorentz66, nsz21, sl68}. 

Throughout the text we adopt the notation $\D$ for the open unit disk and $\T$ for the unit circle. Furthermore, $\|\cdot \|_{\sfE}$ will denote the supremum norm corresponding to a given compact set $\sfE$. For the case where the weight in question is constantly equal to $1$, we adopt the notation $T_n^{\sfE}$ to denote the minimiser of \eqref{eq:chebyshev_equation} which is the Chebyshev polynomial corresponding to $\sfE$.


\subsection{A specific weighted Chebyshev problem}

In \cite[Problem 8.2]{pommerenke72} Hal\'{a}sz proposed to determine
\begin{equation}
    \lambda_n := \max \Big\{|P(0)|: \|P\|_{\T}\leq 1,\, P(1) = 0,\, \deg(P) \leq n\Big\}
    \label{eq:halasz}
\end{equation}
since this can be used to get information about Tur\'{a}n's inequality for lacunary polynomials.

Blatt, see e.g. \cite{lsv79}, suggested to consider the related problem of determining
\begin{equation}
    \mu_n := \min \left\{\|P\|_{\T}:P(z) = (z-1)\prod_{k=1}^{n-1}(z-a_k)\right\}.
    \label{eq:blatt}
\end{equation}
In order to discuss these and related problems, we introduce the function $w_s:\overline{\D}\rightarrow \C$ defined by
\begin{equation}
    w_s(z) = (z-1)^s
\end{equation}
where $s\in [0,\infty)$ and consider $T_n^{w_s} := T_n^{|w_s|}$. The set $\sfE$ in \eqref{eq:chebyshev_equation} defining $T_n^{w_s}$ is in this case $\T$. By definition,
\begin{equation}
    \|w_sT_{n}^{w_s}\|_{\T} = \min_{a_1,\dotsc,a_{n}\in \C}\max_{z\in \T}\left|(z-1)^s\prod_{k=1}^{n}(z-a_k)\right|
    \label{eq:arbitrary_zeros}
\end{equation}
and it is clear from \eqref{eq:blatt} and \eqref{eq:arbitrary_zeros} that
\[\mu_{n+1} = \|w_1T_n^{w_1}\|_{\T}.\]

The quantities $\mu_n$ and $\lambda_n$ can be related using so-called reciprocal polynomials. If 
\[P(z) = c\prod_{k=1}^{n}(z-a_k)\]
then the reciprocal of $P$ is defined as the polynomial
\begin{equation}
    P^\ast(z) = \overline{c}\prod(1-\overline{a}_kz).
    \label{eq:reciprocal}
\end{equation}
These polynomials have the same absolute value on $\T$, that is, $|P(e^{it})| = |P^\ast(e^{it})|$ for any $t\in \R$. Also if $z\neq 0$ and  $P(z) = 0$ then $P(1/\bar{z}) = 0$. It is easily seen that
\[\left(\frac{w_1(z)T_{n-1}^{w_1}(z)}{\|T_{n-1}^{w_1}w_1\|_{\T}}\right)^\ast\]
is a solution to \eqref{eq:halasz} for $n>1$ and thus
\[\mu_n = \frac{1}{\lambda_n}.\]

The exact values of $\lambda_n$ and $\mu_n$ were determined by Lachance, Saff and Varga in \cite{lsv79}. In particular, they showed that
\begin{equation}
    \lambda_n = \left[\cos\left(\frac{\pi}{2(n+1)}\right)\right]^{n+1}.\label{eq:halasz_sol}
\end{equation}This result was achieved by relating $w_{1}T_n^{w_1}$ to weighted Chebyshev polynomials on the interval $[-1,1]$, corresponding to the weight functions $(1-x)$ and $(1-x)(1+x)^{1/2}$. Such polynomials in turn coincide with dilations of the classical Chebyshev polynomials and hence their precise representation is known. Lachance, Saff and Varga substantially generalised the original problem proposed by Hal\'{a}sz and Blatt by showing that one can relate  $w_s T_n^{w_s}$ to weighted Chebyshev polynomials on $[-1,1]$ for any integer value of $s$, a result that we wish to extend to general $s\in (0,\infty)$.

An intermediate step in relating $w_sT_n^{w_s}$ to weighted Chebyshev polynomials on the interval is to consider the following minimisation problem. Find $\alpha_1^\ast\dotsc,\alpha_n^\ast\in [0,2\pi)$ such that
\begin{equation}
    \max_{z\in \T}\left|(z-1)^s\prod_{k=1}^{n}(z-e^{i\alpha_k^\ast})\right| = \min_{\alpha_1,\dotsc,\alpha_n\in [0,2\pi)}\max_{z\in \T}\left|(z-1)^s\prod_{k=1}^{n}(z-e^{i\alpha_k})\right|.
    \label{eq:fixed_zeros}
\end{equation}
The difference between the minimisation problems in \eqref{eq:arbitrary_zeros} and \eqref{eq:fixed_zeros} is that the zeros are restricted to lie on the unit circle in \eqref{eq:fixed_zeros}. It is not clear from classical theory whether a minimiser of \eqref{eq:fixed_zeros} is unique. Our main result is the following theorem.

\begin{theorem}
    For $n\in \N$ and $s\in [0,\infty)$, let $\mathring{T}_n^{w_{s+1}}$ denote any minimiser of \eqref{eq:fixed_zeros} with parameter $s+1$. Then
    \begin{equation}
        \frac{1}{s+n+1}\frac{d}{dz}\left\{w_{s+1}(z) \mathring{T}_n^{w_{s+1}}(z)\right\} = w_s(z)T_n^{w_s}(z)
        \label{eq:lsv_derivative}
    \end{equation}
    and
    \begin{equation}
        \|w_sT_n^{w_s}\|_\T = \frac{1}{2}\|w_{s+1}\mathring{T}_n^{w_{s+1}}\|_\T.
        \label{eq:weight_norm_comparison_general}
    \end{equation}
    Consequently, $\mathring{T}_n^{w_{s}}$ is unique for any $s\in [1,\infty)$.
    \label{thm:generalised_LSV}
\end{theorem}
 This extends work done in \cite{lsv79} where the authors show that Theorem \ref{thm:generalised_LSV} is valid for integer values of $s$. By using a result of Bernstein for determining norm asymptotics of weighted Chebyshev polynomials on an interval, see \cite{bernstein30-31, cer23}, we can compute the asymptotic behaviour of $\|w_sT_n^{w_s}\|_\T$ as $n\rightarrow \infty$. From \cite[Theorem 13]{nsz21}, one can deduce that $\|w_sT_n^{w_s}\|_\T\geq 1$ for any combination of $n$ and $s$. The sequence $\{T_n^{w_s}\}_{n\geq 1}$ asymptotically saturates this theoretical lower bound.
\begin{theorem}
    For any $s\geq 0$ the sequence $n\mapsto \|w_sT_n^{w_s}\|_{\T}$ is monotonically decreasing and
    \[\lim_{n\rightarrow \infty}\|w_sT_n^{w_s}\|_{\T} = 1.\]
    \label{thm:norm_limits}
\end{theorem}

\subsection{Why moving the zeros to the boundary matters}

Our interest in extending Theorem \ref{thm:generalised_LSV} from integer values of $s$ to arbitrary $s>0$ originates from a study of Chebyshev polynomials with respect to the lemniscatic family 
\begin{equation}
    \sfE_m = \{z: |z^m-1| = 1\},\quad m\in \N.
    \label{eq:lemniscate}
\end{equation}
For a graphical representation of these sets, see Figure \ref{fig:lemniscates} below.

\begin{figure}[h!]
\begin{subfigure}{.33\textwidth}
  \centering
  \includegraphics[width=\linewidth]{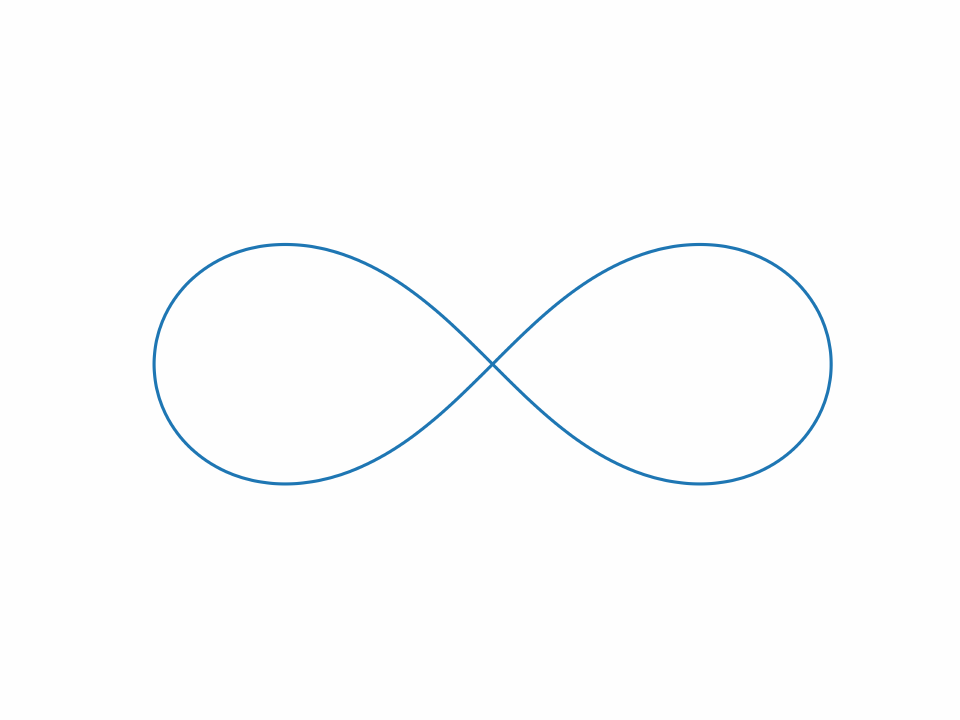}
  \caption{$\sfE_2$}
  \label{fig:lem2}
\end{subfigure}%
\begin{subfigure}{.33\textwidth}
  \centering
  \includegraphics[width=\linewidth]{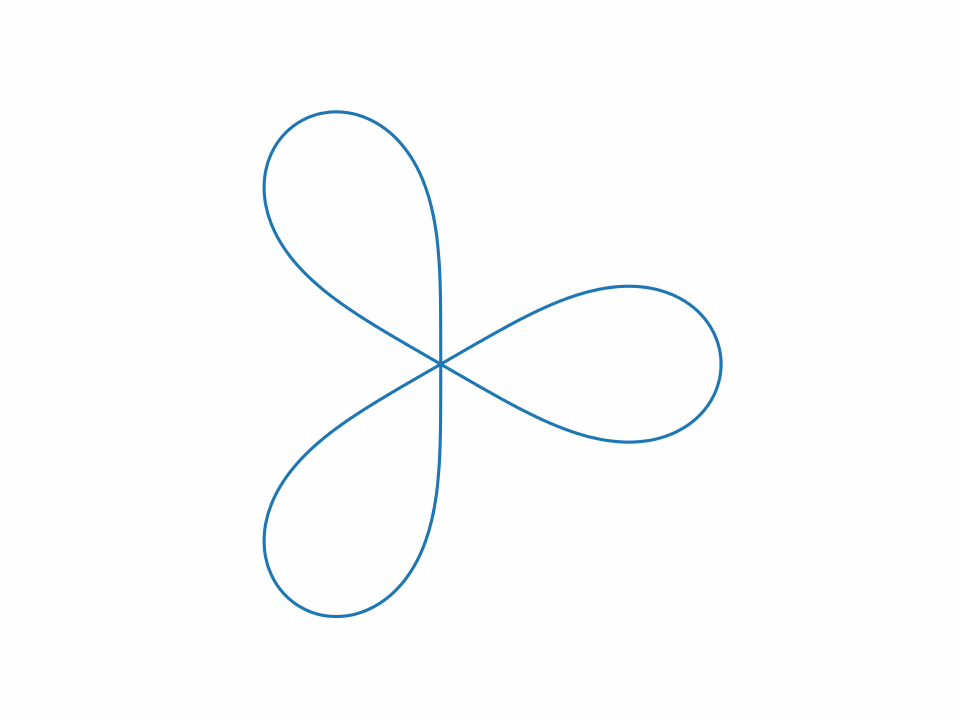}
  \caption{$\sfE_3$}
  \label{fig:lem3}
\end{subfigure}
\begin{subfigure}{.33\textwidth}
  \centering
   \includegraphics[width=\linewidth]{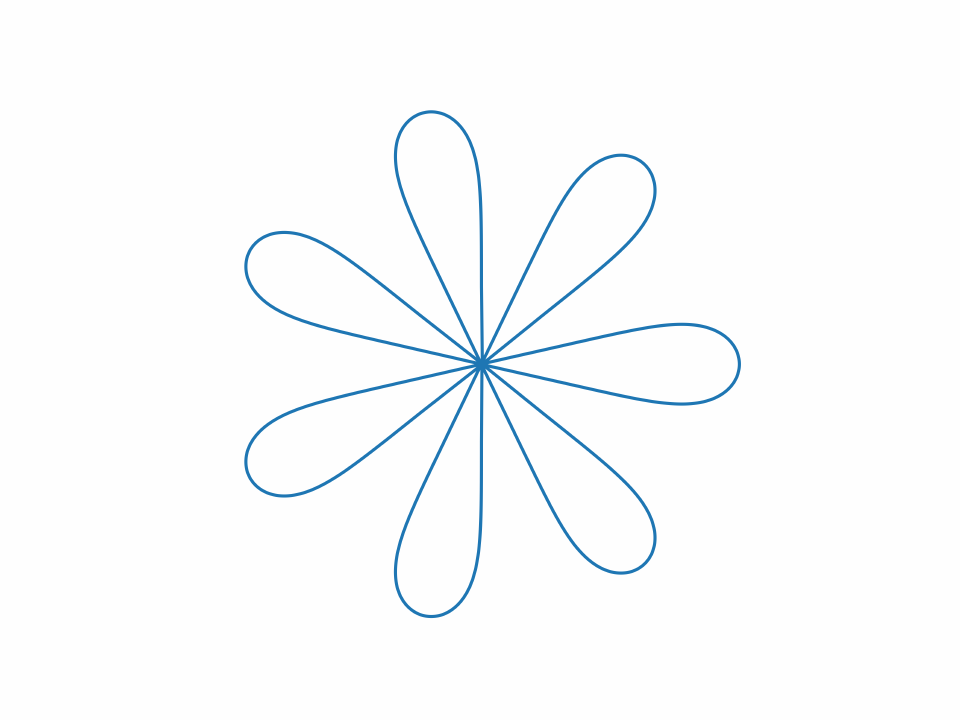}
  \caption{$\sfE_7$}
  \label{fig:lem7}
\end{subfigure}
    \caption{Examples of sets in the family $\{\sfE_m\}$.}
    \label{fig:lemniscates}
\end{figure}

With the notation $T_{nm+l}^{\sfE_m}$, we denote the corresponding Chebyshev polynomial of degree $nm+l$ where $0\leq l <m$. Due to the symmetry of the sets $\sfE_m$, one finds that $z\in \sfE_m$ if and only if $ze^{2\pi ik/m}\in \sfE_m$ for $k\in \mathbb{Z}$. By uniqueness of the Chebyshev polynomial $T_{nm+l}^{\sfE_m}$, this symmetry of the underlying set means that
\begin{equation}
    \label{eq:symmetry_lemniscate}
    e^{2\pi ik l/m}T_{nm+l}^{\sfE_m}(ze^{2\pi ik/m}) = T_{nm+l}^{\sfE_m}(z).
\end{equation}
In other words, $T_{nm+l}^{\sfE_m}$ is a polynomial in $z^{m}$ multiplied with a factor of $z^l$. This explains the choice of degrees to be $nm+l$. By changing the perspective via the transformation $z^m-1 = \zeta\in \T$, one sees that \eqref{eq:symmetry_lemniscate} implies that
\begin{equation}
    T_{nm+l}^{\sfE_m}(z) = (\zeta+1)^{l/m}\prod_{k=1}^{n}(\zeta-a_k) = w_{l/m}(\zeta)T_n^{w_{l/m}}(\zeta)
    \label{eq:lemniscate_cheb}
\end{equation}
where the right-hand side minimises the deviation from zero on the unit-circle. Hence we are led to consider $T_n^{w_{l/m}}$ in order to draw further conclusions on the behaviour of Chebyshev polynomials corresponding to this family of lemniscates. 

As a consequence of Szeg\H{o}'s inequality, see e.g. \cite{szego24} and \cite[Theorem 5.5.4]{ransford95}, one finds that $\|T_{nm+l}^{\sfE_m}\|_{\sfE_m}\geq 1$. This implies that the subsequence $T_{nm}^{\sfE_m}$ is explicitly given by
\begin{equation}
    \label{eq:mnth_degree_cheb}
    T_{nm}^{\sfE_m}(z) = (z^m-1)^n.
\end{equation}
Indeed, these polynomials trivially saturate the lower bound $\|T_{nm}^{\sfE_m}\|_{\sfE_m} = 1$ for any value of $n$. As is evident by \eqref{eq:mnth_degree_cheb}, the zeros of $T_{nm}^{\sfE_m}$ are fixed at the $m$th roots of unity.

For the remaining Chebyshev polynomials $T_{nm+l}^{\sfE_m}$,  $l\in \{1,\dotsc,m-1\}$, very little is known. It is not immediately clear from classical theory what the asymptotic behaviour of $\|T_{nm+l}^{\sfE_m}\|_{\sfE_m}$ should be when $n\rightarrow \infty$. The same lack of understanding persists for the corresponding zero distributions of $T_{nm+l}^{\sfE_m}$.

Extremal polynomials associated with the sets $\sfE_m$ for $m\in \N$ have a long history of study. Already Ullman showed that the zeros of the so-called Faber polynomials for odd degrees corresponding to $\sfE_2$, distribute ``conformally'' about the bounding curve, see \cite{ullman59}. In effect, Ullman's result says that the normalised zero distribution corresponding to the sequence of odd-degree Faber polynomials converge weakly to the equilibrium measure of $\sfE_2$. He obtained this result by first transforming the domain to the unit disk, similar to the transformation in \eqref{eq:lemniscate_cheb}. Peherstorfer and Steinbauer established a connection between different classes of weighted $L^q$ extremal polynomials, $q\in [1,\infty)$, on circles and $\sfE_m$, again using the same change of variables as in \eqref{eq:lemniscate_cheb}, see \cite[Proposition 6]{ps99}. Mi\~{n}a D\'{i}az, in his doctoral thesis, considered Bergman-Carleman polynomials on the closely related sets $\{z: |z^m-1|\leq r\}$ for $r>1$ and determined their complete asymptotical behaviour. In particular, he established that the resulting normalised zero distributions converge weakly to equilibrium measure on $\sfE_m$ for certain subsequences, see \cite[Theorem IV.1.1]{minadiaz06}. The asymptotics for the Bergman-Carleman polynomials corresponding to $\{z:|z^m-1|<r\}$ where $r<1$ were determined by Gustafsson et al., see \cite[Proposition 7.3]{gpss09}.  We return to considerations of Chebyshev polynomials corresponding to the lemniscatic family $\{\sfE_m\}$ in the final section.

In the unweighted case, the Chebyshev polynomials of the unit circle are simply given by $z^n$. That is to say, $T_n^{w_0}(z) = z^n$ which satisfies $\|T_n^{w_0}\|_{\T} = 1$ for all $n$. Interestingly, Chebyshev polynomials corresponding to a weight function consisting of a single prescribed weighted zero on the boundary radically changes the behaviour of the minimiser. 

For $s\geq 0$ and $n\in \N$, we introduce the family of measures
\begin{equation}
    \nu_{n,s} = \frac{1}{n}\sum_{\left\{z\,:\, T_n^{w_s}(z) = 0\right\}}\delta_z
    \label{eq:counting_measure}
\end{equation}
where $\delta_z$ is the Dirac measure at the point $z$ and the sum is taken over all zeros of $T_n^{w_s}$ counting multiplicity. It is clear that $\nu_{n,s}$ is a probability measure for all values of $s$ and $n$. While the zeros of the unweighted polynomials $T_n^{w_0}$ are all situated at $z = 0$, we have the following result for $s = 1$.

\begin{theorem}
    \label{thm:zero_limits}
    With $\nu_{n,s}$ defined as in \eqref{eq:counting_measure} and $0<r<1$, we have
    \[\nu_{n,1}\Big(\{z:|z|\leq r\}\Big)\rightarrow 0\]
    as $n\rightarrow \infty$. Moreover, $\nu_{n,1}$ converges in the weak-star sense to $\frac{d\theta}{2\pi}$ on $\T$ as $n\rightarrow \infty$.
\end{theorem}

Similar behaviour regarding the zero placement for $T_n^{w_s}$ have been observed numerically for other values of $s>0$ and is something that will be further discussed in the concluding section. This should be compared with the Theorem in \cite{st90} stating that for the closure of a Jordan domain $\sfE$, the zeros of the corresponding Chebyshev polynomials $T_n^{\sfE}$ stay away from the boundary if and only if the bounding curve of $\sfE$ is analytic. Although $\T$ is certainly an analytic curve, Theorem \ref{thm:zero_limits} shows that prescribing a zero on the boundary completely changes the behaviour of the zero placement of the corresponding Chebyshev polynomials. In this case, we see that every point of $\T$ is a limit point of zeros of $T_n^{w_s}$.



Another apparent difference is exhibited in terms of the corresponding norm behaviour. Let again $\sfE$ denote the closure of a Jordan domain, such that the associated exterior Riemann map $\Phi:\C\setminus \overline{\D}\rightarrow \C\setminus \sfE$ is of the form $\Phi(z) = z+O(1)$ as $z\rightarrow \infty$ which extends conformally to $\{z:|z|>r\}$ for some $r\in (0,1)$. The existence of such a value $r$ is the same as saying that the bounding curve is analytic. In this case, it was shown by Faber \cite{faber20} that the associated Chebyshev polynomials $T_n^{\sfE}$ satisfy
\begin{equation}
    \|T_n^{\sfE}\|_{\sfE} = 1+o(r^n),\quad n\rightarrow \infty.
    \label{eq:Faber_convergence}
\end{equation}
It is a consequence of the aforementioned inequality of Szeg\H{o} that the conditions on $\Phi$ implies that $\|T_n^{\sfE}\|_{\sfE}\geq 1$ so in this case the theoretical minimum of norms is asymptotically attained.

Even though $\Phi(z) = z$ in the case where $\sfE = \overline{\D}$, we conclude from \eqref{eq:halasz_sol} that
\[\|w_1T_{n-1}^{w_1}\|_{\T} = \mu_n = 1+\frac{\pi^2}{8n}+O(1/n^2)\]
implying that the convergence is significantly slower if a zero is prescribed on the boundary. 

Totik and Varj\'{u} investigated the effects of adding several prescribed zeros on the unit circle, see \cite{tv07}. It is a remarkable result that there exists a universal constant $c>0$ such that if $n_1,n_2$ are integers then
\[\max_{z\in \T}\left|\prod_{k=1}^{n_1}(z-e^{i\alpha_k})\prod_{k=1}^{n_2}(z-a_k)\right|\geq 1+c\cdot \frac{n_1}{n_1+n_2}\]
for any choice of values $\alpha_k\in \R$ and $a_k\in \C$. For further results concerning the addition of several prescribed zeros to Chebyshev polynomials, we refer the reader to \cite{ab10,totik16}.

The main tool used in \cite{lsv79} to prove \eqref{eq:lsv_derivative} for $s\in \N$ is the classical Erd\H{o}s--Lax inequality which holds for polynomials. An extension of this inequality to the present setting constitutes the biggest obstacle in proving Theorem \ref{thm:generalised_LSV}. 

\subsection{A result of Erd\H{o}s--Lax type}

As $k$ ranges between $1,\dotsc,N$, let $\alpha_k\in [0,2\pi)$ and $s_k>0$. We define
\begin{equation}
    f(z) = c\prod_{k=1}^{N} (z-e^{i\alpha_k})^{s_{k}}, \quad z \in \mathbb{D}
    \label{eq:f}
\end{equation}
where the branches are chosen as to make $f$ analytic in $\mathbb{D}$.
The function $f$ extends continuously to the unit circle. Additionally, if $s_k\geq 1$ for all $k\in \{1,\dotsc,N\}$ then the derivative $f'$ also extends continuously to the unit circle. 

With the purpose of proving Theorem \ref{thm:generalised_LSV} we show the following Erd\H{o}s--Lax type equality.
\begin{theorem}\label{thm:erdos_lax_circle}
    Let $f$ be as in \eqref{eq:f} with $s_k\geq 1$. Then
    \begin{equation*}
        \norm{f'}_{\T} = \frac{\sum_{k=1}^{N} s_{k}}{2} \norm{f}_{\T}.
    \end{equation*}
\end{theorem}

The classical Erd\H{o}s--Lax inequality was first proven in \cite{lax44}, see Section \ref{section:el} for further details. Theorem \ref{thm:erdos_lax_circle} shows that the case of equality is not specific to polynomials; the result can be suitably extended to allow for powers of polynomials. We believe that Theorem \ref{thm:erdos_lax_circle} is of independent interest and therefore also present related consequences that our result provides. In particular, we prove a generalisation of the classical Erd\H{o}s--Lax inequality which allows for the function $f$ to have zeros away from the unit circle in addition to the zeros on the boundary.

Since Lax's proof several articles on the Erd\H{o}s--Lax inequality have been written. Some of these have provided new ways of proof and others have generalised the result, see e.g. \cite{az80, malik69}. However, to the authors' knowledge, none have considered non-integer powers of polynomials as in Theorem \ref{thm:erdos_lax_circle}. 

\subsection{Outline}
In Section \ref{section:el} we prove Theorem \ref{thm:erdos_lax_circle} and discuss how it relates to classical results for polynomials. We further extend the equality of Theorem \ref{thm:erdos_lax_circle} to an inequality which allows for zeros of the functions to lie outside the closed unit disk. As a consequence, we show that the classical Erd\H{o}s--Lax inequality can be extended beyond the setting of polynomials to powers of polynomials.

In Section \ref{section:lsv} we use Theorem \ref{thm:erdos_lax_circle} to extend the results by Lachance, Saff and Varga \cite{lsv79} and prove Theorems \ref{thm:generalised_LSV} and \ref{thm:norm_limits}. As an intermediate step we relate the Chebyshev polynomials corresponding to the weight $w_s$ on the unit circle to Chebyshev polynomials corresponding to Jacobi weights on the interval. We end this section with a proof of Theorem \ref{thm:zero_limits}.

In Section \ref{sec:conclusion} we discuss how the results for Chebyshev polynomials corresponding to a weight with a prescribed zero can be applied to draw conclusions for Chebyshev polynomials on lemniscates. We further present numerical experiments indicating that Theorem \ref{thm:zero_limits} should be valid for all positive values of the parameter $s$. In other words, it appears that the normalised zero counting measure $\nu_{n,s}$ defined in \eqref{eq:counting_measure} converges weak-star as $n\rightarrow \infty$ to normalised Lebesgue measure on $\T$ for all values of $s>0$. Finally, we present a version of the generalised Erd\H{o}s--Lax inequality that we believe should hold true.

\begin{section}{An Erd\H{o}s--Lax type equality for powers of polynomials}
\label{section:el}
    If $P$ is a polynomial of degree $n$, then Bernstein showed that
\begin{equation} 
\label{eq:bernstein}
\|P'\|_{\T}\leq n\|P\|_{\T}
\end{equation}
with equality if and only if $P(z) = cz^n$ for some constant $c\in \mathbb{C}$, see e.g. \cite{bernstein12, qz19}. Erd\H{o}s suggested that this inequality can be sharpened in the case where $P$ is zero free on $\mathbb{D}$. Specifically, he conjectured that the sharp inequality in this case is given by
\begin{equation} 
\label{eq:erdos-lax}
\|P'\|_{\T}\leq \frac{n}{2}\|P\|_{\T}.
\end{equation}
This was later proven to be correct by Lax \cite{lax44} using techniques developed by P\'{o}lya and Szeg\H{o}. P\'{o}lya and Szeg\H{o} verified the special case of equality in \eqref{eq:erdos-lax} when the zeros of $P$ are situated on $\mathbb{T}$. 
Lax was able to lift this result using reciprocal polynomials to prove \eqref{eq:erdos-lax}. We should mention that preceding Lax's proof of \eqref{eq:erdos-lax}, Tur\'{a}n \cite{turan39} showed that if $P$ is a polynomial which is zero free on $\mathbb{C}\setminus \mathbb{D}$ then
\begin{equation}
    \|P'\|_{\T} \geq \frac{n}{2}\|P\|_{\T}.
    \label{eq:turan}
\end{equation}

Here we generalise the case of equality in \eqref{eq:erdos-lax} that was proven by P\'{o}lya and Szeg\H{o} to the setting where, instead of a polynomial $P$, one considers the function $f = P^{1/m}$. We show that one can replace the degree $n$ in the equation with the number $n/m$. This will establish Theorem \ref{thm:erdos_lax_circle} in the case where all the values of $s_k$ are rational numbers. The fact that the result is valid for all values of $s_k\geq 1$ is a consequence of the rational numbers being dense among the real numbers.


\subsection{Rational powers of polynomials}

With the intent of proving Theorem \ref{thm:erdos_lax_circle} we make a first reduction of the problem by noting that it is enough to consider rational values of $s_k$ by density. Denote by $P$ the polynomial
\begin{equation}
    \label{eq:polynomial}
    P(z) = c^m\prod_{k=1}^{N} (z-a_k)^{s_{k}m} = c^m\prod_{k=1}^{n}(z-c_k)
\end{equation}
where the $s_k\geq 1$ are rational numbers, and $m \geq 1$ is an integer such that $ms_{k} \in \mathbb{N}$ for all $k$. The $c_{k}$'s are the zeros of $P$ counting multiplicities. In particular, for every $k\in \{1,\dotsc,n\}$ there exists a $j\in \{1,\dotsc,N\}$ such that $c_k = a_j$. The reason for introducing both notations is that it will be advantageous to be able to refer to the different representations later. Note that the degree of $P$ is given by $n =m\sum_{k=1}^{N} s_{k}$. 

Let 
\begin{equation}\label{eq:form}
    f(z) = P(z)^{1/m} = c\prod_{k=1}^{N} (z-a_k)^{s_{k}}, \quad z \in \mathbb{D}
\end{equation}
which is a multi-valued function with single-valued absolute value. If all $|a_k|\geq 1$ then we choose a branch so as to make $f$ analytic in $\mathbb{D}$. The derivative of $f$ is given by
\begin{equation}\label{eq:derivative}
    f'(z)=\frac{1}{m}f(z)\frac{P'(z)}{P(z)} = \frac{f(z)}{m} \sum_{k=1}^{n} \frac{1}{z-c_{k}}.
\end{equation}
If all $c_k$ lie exterior to $\D$, then $f'$ is continuous in $\overline{\D}$.

    
    The restriction $s_k\geq 1$ is necessary. If this is not the case, then the function $f'$ becomes unbounded on $\mathbb{T}$ if $|a_k| = 1$ for some $k$. Secondly, note that if $f(z) = (z-a)^s$ with $|a|\geq 1$ and $s\geq 1$, then
	\[f'(z) = s(z-a)^{s-1},\]
	and thus $\|f'\|_\T = s(|a|+1)^{s-1}$. Since $\|f\|_\T = (|a|+1)^s$, this implies that
	\[\frac{\|f'\|_\T}{\|f\|_\T} = \frac{s}{|a|+1}\leq \frac{s}{2}\]
	with equality if and only if $|a|=1$. Also note that $f'$ achieves its maximum modulus in this case at the same point on the unit circle as $f$ does. This is something we will investigate further. 
 
 The key step in the proof of Theorem \ref{thm:erdos_lax_circle} below will be to show that there is a point $z_0\in \T$ such that 
    \[\|f'\|_{\T} = |f'(z_0)| = \frac{\sum_{k=1}^{N} s_k}{2}|f(z_0)|.\]
    This trivially implies that
    \[\|f'\|_{\T}\leq \frac{\sum_{k=1}^{N} s_k}{2}\|f\|_\T.\]
    We isolate the reverse inequality which is a generalisation of \eqref{eq:turan} from \cite{turan39} in a separate lemma since it is considerably easier to show.

     \begin{lemma}
        \label{lem:turan}
        Let $f(z) = c\prod_{k=1}^{N}(z-a_k)^{s_k}$ with $s_k\geq 1$ and $|a_k|\leq 1$. Then
        \[\|f'\|_{\T}\geq \frac{\sum_{k=1}^{N} s_k}{2}\|f\|_{\T}.\]
    \end{lemma}
    \begin{proof}
        By differentiating, we obtain
        \[f'(z) = f(z)\sum_{k=1}^{N}\frac{s_k}{z-a_k}\]
        and the result follows if we can show that
        \[\left|\sum_{k=1}^{N}\frac{s_k}{z-a_k}\right|\geq \frac{\sum_{k=1}^{N} s_k}{2}\quad \text{for }z\in \mathbb{T}.\]
        To see this, note that, with $|z|=1$,
        \[\text{Re}\sum_{k=1}^{N}s_k\frac{z}{z-a_k}
            = \frac{1}{2}\sum_{k=1}^{N}s_k\frac{2-2\text{Re}(z\overline{a_k})}{1-2\text{Re}(z\overline{a_k})+|a_k|^2} \geq \frac{1}{2}\sum_{k=1}^{N}s_k \qedhere
\popQED \]
    \end{proof}
    
 In Theorem \ref{thm:erdos_lax_circle}, attention is restricted to the case where all $a_k$ lie on the unit circle. For this reason we introduce the notation $a_k = e^{i\alpha_k}$ for some $\alpha_k\in [0,2\pi)$. It then follows that
 \[
    f(z) = P(z)^{1/m} = c\prod_{k=1}^{N} (z-e^{i\alpha_k})^{s_{k}}
 \]
 The argument used to prove Theorem \ref{thm:erdos_lax_circle} will be an adaptation of Szeg\H{o}'s argument for polynomials as it appears in \cite{lax44}. Of course, suitable modifications are required to conform to the present situation. Specifically, we prove that any point of maximality for $f'$ on $\mathbb{T}$ is a point of maximality for $f$ with one exception.

    \begin{proof}[Proof of Theorem \ref{thm:erdos_lax_circle}]
    Recall that we are assuming $s_{k} \geq 1$ for all $k$ and that $\sum_{k=1}^{N} s_k = n/m$ where $n$ is the degree of the polynomial $P$ satisfying $f(z) = P(z)^{1/m}$. As we already remarked, these restrictions on $s_{k}$ imply that $f'$ is continuous in $\overline{\D}$. By \eqref{eq:derivative}, the second derivative of $f$ is given by
\begin{equation*}
    \begin{split}
    f''(z)& = \frac{1}{m}\left( f'(z)\frac{P'(z)}{P(z)} + f(z)\frac{P''(z)P(z)-(P'(z))^{2}}{P(z)^2}\right)\\& = \frac{1}{m}\left( f'(z)\frac{P'(z)}{P(z)} + mf'(z)\frac{P(z)}{P'(z)}\frac{P''(z)P(z)-(P'(z))^{2}}{P(z)^2}\right) \\ &= \frac{f'(z)}{m}\left( (1-m)\frac{P'(z)}{P(z)} + m\frac{P''(z)}{P'(z)}\right).
    \end{split}
\end{equation*}
With the intent of proving that $\|f'\|_\T\leq \frac{n/m}{2}\|f\|_\T$, we let $z_{0} \in \mathbb{T}$ be a point where $|f'|$ assumes its maximal value on $\mathbb{T}$. Consider first the case where $P(z_{0}) \neq 0$ and hence also $P'(z_{0}) \neq 0$ (otherwise $f'(z_0) =0$). The derivative of the map $\theta \mapsto \lvert f'(z_{0}e^{i\theta}) \rvert^{2}$ attains the value $0$ at $\theta = 0$. Thus
\begin{equation*}
    z_{0}f''(z_{0})\overline{f'(z_{0})} = \overline{z_{0}}f'(z_{0})\overline{f''(z_{0})},
\end{equation*}
meaning that $z_{0}f''(z_{0})/f'(z_{0}) \in \mathbb{R}$. Combining this with the relation for $f''$, we obtain
\begin{equation}\label{eq:max}
    \mathbb{R} \ni z_{0}\cdot \frac{f''(z_{0})}{f'(z_{0})} = \frac{1}{m}\left( (1-m)\text{Re}\left(z_{0}\frac{P'(z_{0})}{P(z_{0})}\right) + m\text{Re}\left(z_{0}\frac{P''(z_{0})}{P'(z_{0})}\right)\right).
\end{equation}
For $z \in \mathbb{T}$, recalling the representation of $P$ from \eqref{eq:polynomial}, we see that
\begin{equation}
    \label{eq:prime_over_poly}
    z\cdot \frac{P'(z)}{P(z)} = \sum_{k=1}^{n} \frac{z}{z-c_{k}} = \frac{n}{2} + i \sum_{k=1}^{n} t_{k}(z),
\end{equation}
where $t_{k}(z) = \text{Im}(z(z-c_{k})^{-1})$. If $\sum_{k=1}^{n} t_{k}(z_0) = 0$, then we obtain from equation \eqref{eq:derivative} the inequality
\[\|f'\|_\T = |f'(z_0)| = \frac{n/m}{2}|f(z_0)|\leq \frac{n/m}{2}\|f\|_\T.\]
We want to show that this is the only possibility with one exception. In this exceptional case, the inequality $\|f'\|\leq \frac{n/m}{2}\|f\|$ can easily be deduced as we shall see. Assume therefore that $\sum_{k=1}^{n} t_{k}(z_0) \neq 0$. We consider the quotients
\begin{equation}
    z^{2}\cdot\frac{P''(z)}{P(z)} = \left( \sum_{k=1}^{n} \frac{z}{z-c_{k}}\right)^{2} - \sum_{k=1}^{n} \left( \frac{z}{z-c_{k}}\right)^{2}
    \label{eq:biss_over_poly}
\end{equation}
and
\begin{equation}
    z\cdot\frac{P''(z)}{P'(z)} = z^{2}\frac{P''(z)}{P(z)} / z\frac{P'(z)}{P(z)} = \sum_{k=1}^{n} \frac{z}{z-c_{k}} - \frac{\sum_{k=1}^{n} \left( \frac{z}{z-c_{k}}\right)^{2}}{\sum_{k=1}^{n} \frac{z}{z-c_{k}}}.
    \label{eq:biss_over_prime}
\end{equation}
By substituting the explicit representations of \eqref{eq:prime_over_poly} and \eqref{eq:biss_over_prime} into \eqref{eq:max}, we see that
\begin{align*}
    &(1-m)\sum_{k=1}^{n} t_{k}(z_0)\\ + &m\left( \sum_{k=1}^{n} t_{k}(z_0) - \frac{n/2\sum_{k=1}^{n} t_{k}(z_0) - \sum_{k=1}^{n} t_{k}(z_0) \sum_{k=1}^{n} (1/4-t_{k}(z_0)^{2})}{n^{2}/4 + \left( \sum_{k=1}^{n} t_{k}(z_0) \right)^{2} }\right) = 0.
\end{align*}
Rewriting this expression and using \eqref{eq:prime_over_poly}, it follows that
\begin{equation}
    \left|\sum_{k=1}^{n}\frac{z}{z-c_k}\right|^2 = \frac{n^{2}}{4} + \left( \sum_{k=1}^{n} t_{k}(z_0) \right)^{2} = m\left( n/2-\sum_{k=1}^{n} (1/4-t_{k}(z_0)^{2}) \right).
    \label{eq:reciprocal_zero_sum}
\end{equation}
Inserting the expression for $\left|\sum_{k=1}^{n}\frac{z}{z-c_k}\right|$ from \eqref{eq:reciprocal_zero_sum} into \eqref{eq:biss_over_prime} results in
\begin{equation*}
    \begin{split}
        \text{Re}\left(z_{0}\frac{P''(z_0)}{P'(z_0)}\right) = \frac{n}{2} - \frac{n/2\sum_{k=1}^{n}(1/4-t_{k}(z_0)^{2}) + \left(\sum_{k=1}^{n} t_{k}(z_0) \right)^{2}}{n^{2}/4+\left( \sum_{k=1}^{n} t_{k}(z_0) \right)^{2}} \\ = \frac{n}{2} - \frac{n/2\left( n/2 - n^{2}/4m -  \left( \sum_{k=1}^{n} t_{k}(z_0) \right)^{2}/m\right) + \left( \sum_{k=1}^{n} t_{k}(z_0) \right)^{2}}{n^{2}/4+\left( \sum_{k=1}^{n} t_{k}(z_0) \right)^{2}} \\ = \frac{n}{2} - \frac{(1-n/2m)n^{2}/4 + (1-n/2m)\left( \sum_{k=1}^{n} t_{k}(z_0) \right)^{2}}{n^{2}/4+\left( \sum_{k=1}^{n} t_{k}(z_0) \right)^{2}} = \frac{n}{2}\left(1+\frac{1}{m}\right)-1.
    \end{split}
\end{equation*}
    Finally, we obtain that
    \begin{equation*}
        z_{0}\frac{f''(z_{0})}{f'(z_{0})} = \frac{1}{m}\left( (1-m)\frac{n}{2} + m\left( \frac{n}{2}\left(1+\frac{1}{m}\right)-1\right) \right) = \frac{n}{m}-1,
    \end{equation*}
    whenever $z_{0}$ is a maximum of $f'$ such that $\sum t_k(z_0)\neq 0$. From \eqref{eq:form} one realises that $(f')^{m} = Q$ is a polynomial of degree $n-m$. Thus
    \begin{equation*}
        \frac{z_0}{m}\frac{Q'(z_{0})}{Q(z_{0})} = z_{0}\frac{f''(z_{0})}{f'(z_{0})} = \frac{n}{m}-1.
    \end{equation*}
    Since $Q$ attains its maximum modulus on the unit circle at the same point, $z_{0}$, as $f'$, it follows from Bernstein's inequality that
    \[\|Q'\|_\T = (n-m)\|Q\|_\T.\]
    Since equality holds, this further implies that $Q(z) = c^mz^{n-m}$. By taking the $m$th root we find that $f' = cz^{n/m-1}$ (also, $n/m$ must be an integer). In this case, one sees that
    \begin{equation}
        f(z) = \tilde c(z^{n/m}+e^{i\alpha})
        \label{eq:bernstein_equality_case}
    \end{equation}
    for some $\alpha\in [0,2\pi)$ and it is verified by direct computation that the inequality $\|f'\|\leq \frac{n/m}{2}\|f\|$ holds in this case. 
    
    With this we have proven that $\sum_{k=1}^{n}t_k(z_0)\neq 0$ implies that $f$ has the representation \eqref{eq:bernstein_equality_case} for which the desired inequality can be shown to hold. In the remaining case we have that
    \begin{equation*}
        \sum_{k=1}^{n} t_{k}(z_{0}) = 0
    \end{equation*}
    and using \eqref{eq:prime_over_poly} we conclude that
    \begin{equation*}
        \|f'\|_\T = \lvert f'(z_{0}) \rvert = \frac{n/m}{2} \lvert f(z_{0}) \rvert \leq \frac{n/m}{2}\|f\|_\T.
    \end{equation*}
    It remains to consider the case when $P(z_{0}) = 0$. Then $z_{0} = e^{i\alpha_k}$ for some $k$ and $s_{k} = 1$, since otherwise $f'(z_{0}) = 0$. Then the right-hand side of
    \begin{equation*}
        \frac{f''(z)}{f'(z)} = \frac{1}{m}\left( (1-m)\frac{P'(z)}{P(z)} + m\frac{P''(z)}{P'(z)}\right)
    \end{equation*}
    has a singularity at $e^{i\alpha_k}$ (unless $m = 1$, in which case it is nothing but the classical version of the theorem), but the left-hand side does not since $f''$ is continuous at $e^{i\alpha_k}$ and $f'$ has a maximum. Hence this cannot happen (unless $f' \equiv 0$).

    As the lower bound $\|f'\|_\T\geq \frac{n/m}{2}\|f\|_\T$ follows from Lemma \ref{lem:turan} the proof is complete.
    \end{proof}

    From the proof of Theorem \ref{thm:erdos_lax_circle} we gather the following corollary.

    \begin{corollary}
        Let $$f(z) = c\prod_{k=1}^{N}(z-a_k)^{s_k}$$ with $|a_k| = 1$ and $s_k\geq 1$. Then $|f'|$
		attains its maximum on the unit circle at any point $z$ for which $|f(z)| = \|f\|_\T$, and in that case 
		\[
			\left|\sum_{k=1}^{N}\frac{s_k}{z-a_k}\right| = \mathrm{Re}\sum_{k=1}^{N}s_k\frac{z}{z-a_k}= \frac{\sum_{k=1}^{N}s_k}{2}.
		\]
    \end{corollary}
    Even though it is not difficult to conclude this result for polynomials using Lax's argument, we have not been able to find it in the literature, even for the case where $f$ is a polynomial.
    
    Using the equality of Theorem \ref{thm:erdos_lax_circle}, we are able to extend the classical Erd\H{o}s--Lax inequality to the setting of powers of polynomials. Certain restrictions are needed for the zeros of $f$ in $\left\{ z : \lvert z \rvert > 1\right\}$ compared with Theorem \ref{thm:erdos_lax_circle}.
\begin{theorem}\label{thm:generalised_erdos_lax}
    Let $|a_k|\geq 1$ and let $s_{k} \geq 1$. If
    \begin{equation*}
        f(z) = c \prod_{k=1}^{N} (z-a_{k})^{s_{k}},
    \end{equation*}
    and $s_k\in \mathbb{N}$ whenever $|a_k|>1$, then
    \begin{equation*}
        \|f'\|_\T \leq \frac{\sum_{k=1}^{N} s_{k}}{2} \|f\|_\T
    \end{equation*}
    with equality if and only if all $a_k$'s are unimodular.
\end{theorem}
This result allows us to have zeros of $f$ away from the unit circle. However, the restriction is that the zeros which are located away from the circle must have integer exponents.

Before turning to the proof, we consider reciprocals of $m$th roots of polynomials. With $f$ defined via \eqref{eq:form}, we let $f^\ast(z)$ denote the function
\[f^\ast(z) = \overline{c}\prod_k(1-\overline{a}_kz)^{s_k}.\]
This definition coincides with the $m$th root of the reciprocal from \eqref{eq:reciprocal} corresponding to
\[P(z) = f(z)^m = c^m\prod_{k=1}^N(z-a_k)^{ms_k}.
\]
Indeed, since
\[P^\ast(z) = \overline{c}^m\prod_{k=1}^N(1-\overline{a}_kz)^{ms_k},\]
we see that $f^\ast(z) = P^\ast(z)^{1/m}$. Consequently, $|f(z)| = |f^\ast(z)|$ for $z\in \T$. If all zeros of $P$ are exterior to $\D$, then \cite[Lemma 3.3]{qz19} implies that $|P'(z)|\leq |(P^\ast)'(z)|$ for $z\in \T$. If we therefore assume that $|a_k|\geq 1$, then
\begin{equation}
    |(f^\ast)'(z)| = \left|\frac{1}{m}P^\ast(z)^{1/m-1}(P^\ast)'(z)\right|\geq \left|\frac{1}{m}P^{1/m-1}(z)P'(z)\right| = |f'(z)|.
    \label{eq:reciprocal_derivative}
\end{equation}

We isolate the following lemma which is a generalisation of a result of P\'{o}lya and Szeg\H{o} \cite[v.1, p. 108]{polya-szego76} to the present setting of rational powers of polynomials.

\begin{lemma}
    Let $f(z) = c\prod_{k=1}^{N}(z-a_k)^{s_k}$ with $|a_k|\geq 1$. If $|\zeta|=1$ and $l\in \N$ then for any choice of branch of $f$ and $f^\ast$, the function
    \[f(z)+\zeta z^lf^\ast(z)\]
    is zero free when $z\notin \T$.
    \label{lem:polya_szego_zeros}
\end{lemma}
\begin{proof}
In order to derive a contradiction assume that $f(z) + \zeta z^l f^\ast(z) = 0$ for some $z\notin \T$. In particular, it must be so that $|f(z)| = |z|^l|f^\ast(z)|$. Note that the absolute value is single-valued. The function
\[\log \left|\frac{z^lf^\ast(z)}{f(z)}\right|\]
is subharmonic on $\mathbb{D}$ and superharmonic on $\mathbb{C}\setminus\mathbb{D}$ while constantly equal to zero on the unit circle. Hence the maximum principle \cite[Theorem 2.3.1]{ransford95} implies that the logarithmic expression must be zero free away from $\T$ unless it is in fact a constant function. 
\end{proof}

With this result at hand, we can mimic Lax's technique of proving the Erd\H{o}s--Lax inequality from the corresponding version of Theorem \ref{thm:erdos_lax_circle} for polynomials.

\begin{proof}[Proof of Theorem \ref{thm:generalised_erdos_lax}]
By assumption, $|a_k|\geq 1$ and $s_{k} \geq 1$ are rational numbers. The function 
\[f(z) = c \prod_{k} (z-a_{k})^{s_{k}}\]
is defined and analytic on $\D$. Let $k_j$ denote those indices that correspond to the $a_k$'s which lie on the unit circle. For any of the remaining indices, it therefore holds that $s_k\in \N$. This implies that
\[f^\ast(z) =  \prod_{j}(z-a_{k_j})^{s_{k_j}}Q(z)\]
where $Q$ is a polynomial whose zeros are situated in $\mathbb{D}$. For any $|\zeta|=1$, it holds true that
\[f(z)+\zeta f^\ast(z) = \prod_{j}(z-a_{k_j})^{s_{k_j}}\tilde{Q}(z),\]
where Lemma \ref{lem:polya_szego_zeros} implies that $\tilde{Q}$ is a polynomial which has all its zeros on the unit circle. Therefore we can apply Theorem \ref{thm:erdos_lax_circle} to the function $g = f+\zeta f^\ast$ where $\zeta$ is chosen such that $\arg f'(z_0) = \arg \zeta (f^\ast)'(z_0)$ for a point $z_0$ satisfying $|f'(z_0)| = \|f'\|_\T$. But then
\[\|g'\|_\T\geq |g'(z_0)| = |f'(z_0)|+|(f^\ast)'(z_0)|\geq 2 |f'(z_0)| = 2\|f'\|_\T.\]
Since we further have that \[\|g'\|_\T= \frac{\sum_{k=1}^{N} s_k}{2}\|g\|_\T\leq \left(\sum_{k=1}^{N} s_k \right)\|f\|_\T,\] the result follows.
\end{proof}
\end{section}

\begin{section}{Weighted Chebyshev polynomials on the unit circle with a vanishing weight}
\label{section:lsv}
In this section we prove Theorem \ref{thm:generalised_LSV} and consider some of its consequences. Our method is inspired by the one presented in \cite{lsv79} for the case of $s\in \N$. The proof differs, since it relies on Theorem \ref{thm:erdos_lax_circle} rather than the classical Erd\H{o}s--Lax inequality. 

\begin{proof}[Proof of Theorem \ref{thm:generalised_LSV}]
    Given $s\in (0,\infty)$ and $n\in \N$, let $T_n^{w_s}$ be the minimiser of \eqref{eq:arbitrary_zeros}. Due to the convexity of $\overline{\D}$, it follows that all zeros of $T_n^{w_s}$ lie in $\overline{\D}$. This can be deduced from \cite[Theorem III.3.4]{st97}. Consequently, Lemma \ref{lem:polya_szego_zeros} implies that
    \[zT_n^{w_s}(z)+\zeta (T_n^{w_s})^\ast(z)\]
    has all its zeros on the unit circle for any $|\zeta|=1$. Let $\zeta = -1$ so that we can write
    \[zT_n^{w_s}(z)-(T_n^{w_s})^\ast(z) = (z-1)\prod_{k=1}^{n}(z-e^{i\alpha_k^\ast}).\]
    The fact that there is a zero at $z=1$ is a consequence of $T_n^{w_s}$ having real coefficients. We introduce the polynomial
    \[P(z) = \prod_{k=1}^{n}(z-e^{i\alpha_k^\ast})\]
    whose zeros are all situated on the unit circle. Furthermore,
    \[\|w_{s+1}P\|_{\T} = \max_{z\in \T}\Big|(z-1)^s\left(zT_n^{w_s}(z)-(T_n^{w_s})^\ast(z)\right)\Big|\leq 2\|w_sT_n^{w_s}\|_\T.\]
     Note that $P$ is a candidate for a solution to \eqref{eq:fixed_zeros} for $w_{s+1}$. Let $\mathring{T}_n^{w_{s+1}}$ denote any such minimiser. Then, by definition,
    \[\|w_{s+1}\mathring{T}_n^{w_{s+1}}\|_\T\leq \|w_{s+1}P\|_\T.\]
    On the other hand, by Theorem \ref{thm:erdos_lax_circle},
    \begin{align}
        \label{eq:inequalities}
        \begin{split}
            \|w_sT_n^{w_s}\|_\T&\leq \frac{1}{s+1+n}\left\|\frac{d}{dz}\left(w_{s+1}\mathring{T}_n^{w_{s+1}}\right)\right\|_\T = \frac{1}{2}\|w_{s+1}\mathring{T}_n^{w_{s+1}}\|_\T\\&\leq \frac{1}{2}\|w_{s+1}P\|_\T\leq \|w_sT_n^{w_s}\|_\T.
        \end{split}
    \end{align}
    By uniqueness of the minimiser $w_sT_n^{w_s}$, equality must hold throughout. This implies that
    \begin{align}
        \begin{split}
            w_s(z)T_n^{w_s}(z) & = \frac{1}{s+n+1}\frac{d}{dz}\left(w_{s+1}(z)\mathring{T}_n^{w_{s+1}}(z)\right),
        \end{split}
    \end{align}
    or equivalently that 
    \begin{equation}
        w_{s+1}(z)\mathring{T}_n^{w_{s+1}}(z) = (s+n+1)\int_{1}^{z}w_s(\zeta)T_n^{w_s}(\zeta)d\zeta.
    \end{equation}
    Finally, \eqref{eq:inequalities} immediately implies 
    \[\|w_sT_n^{w_s}\|_{\T} = \frac{1}{2}\|w_{s+1}\mathring{T}_n^{w_{s+1}}\|_\T\]
    which completes the proof.
\end{proof}

In order to prove Theorem \ref{thm:norm_limits}, we first show that there is a way of relating $\mathring{T}_{n}^{w_{s+1}}$ to weighted Chebyshev polynomials on the interval $[-1,1]$ with Jacobi weights. This matter is of independent interest and is something that was also considered in \cite{lsv79} for integer values of $s$. To that end, we introduce for $\alpha,\beta\geq 0$ the weights
\[w^{(\alpha,\beta)}(x) = (1-x)^\alpha(1+x)^{\beta},\quad x\in [-1,1]\]
and the notation $T_{n}^{(\alpha,\beta)}:= T_n^{w^{(\alpha,\beta)}}$ which is reminiscent of the notation typically used for Jacobi polynomials. 

\begin{theorem}
    For $s\geq 1$ and $n\in \N$, we have
    \begin{equation}
        w_s(z)\mathring{T}_{n}^{w_{s}}(z) = \begin{cases}
            (-1)^{\frac{s}{2}}(2z)^{\frac{s+n}{2}}w^{(\frac{s}{2},0)}(x)T_{\frac{n}{2}}^{\left(\frac{s}{2},0\right)}(x),& n\text{ even,} \\
            (-1)^{\frac{s}{2}}(2z)^{\frac{s+n}{2}}w^{(\frac{s}{2},\frac{1}{2})}(x)T_{\frac{n-1}{2}}^{\left(\frac{s}{2},\frac{1}{2}\right)}(x) & n\text{ odd,}
        \end{cases}
        \label{eq:jacobi}
    \end{equation}
    where $x = (z+z^{-1})/2$ and $|z|\leq 1$. As a consequence,
    \begin{equation}
        \|w_s\mathring{T}_n^{w_s}\|_{\T} = \begin{cases}
            2^{\frac{s+n}{2}}\|w^{(\frac{s}{2},0)}T_{\frac{n}{2}}^{(\frac{s}{2},0)}\|_{[-1,1]}& n \text{ even,}\\
            2^{\frac{s+n}{2}}\|w^{(\frac{s}{2},\frac{1}{2})}T_{\frac{n-1}{2}}^{(\frac{s}{2},\frac{1}{2})}\|_{[-1,1]}& n \text{ odd.}
        \end{cases}
        \label{eq:jacobi_norms}
    \end{equation}
    \label{thm:jacobi}
\end{theorem}
\begin{proof}
    By definition, $w_s\mathring{T}_n^{w_s}$ minimises the expression
    \[\max_{|z| = 1}\left|(z-1)^s\prod_{k=1}^{n}(z-e^{i\alpha_k})\right|\]
    over $\alpha_k\in [0,2\pi)$. We choose $\alpha_k^\ast$ such that
    \[\mathring{T}_n^{w_s}(z) = \prod_{k=1}^{n}(z-e^{i\alpha_k^\ast})\]
    and consider the two cases separately. 

    $\mathbf{n}$ \textbf{even:} Write $n = 2m$ where $m$ is a positive integer. Since $\mathring{T}_n^{w_s}$ is the unique minimiser, it follows by symmetry of $w_s$ and $\T$ that all zeros of $\mathring{T}_n^{w_s}$ come in conjugate pairs. By reordering, we may assume that $\alpha_k^\ast = 2\pi-\alpha_{k+m}^\ast$. The map $J(z) = (z+z^{-1})/2$ is a conformal map from $\C\setminus \overline{\D}$ to $\C\setminus [-1,1]$ and furthermore satisfies $J(z) = \mathrm{Re}(z)$ for $z\in \T$. Also, by writing $x = \mathrm{Re}(z) = J(z)$ for $z\in \T$, it is clear that $z^2+1 = 2zx$ and therefore
    \[(z-e^{i\alpha})(z-e^{-i\alpha}) = z^2-2z\cos(\alpha)+1 = 2z(x-\cos \alpha).\]
    In particular, by letting $\alpha = 0$ we see that
    \[(z-1)^2 = 2z(x-1).\]
    We conclude that for $|z| = 1$ the $\alpha_k^\ast$ minimise
    \begin{align*}
        \max_{z\in \T}\left|(z-1)^{s}\prod_{k=1}^{2m}(z-e^{i\alpha_k^\ast})\right| & = \max_{z\in \T}\left|(z-1)^s\prod_{k=1}^{m}2z(x-\cos\alpha_k^\ast)\right| \\
        & = \max_{z\in \T}\left|[(2z)(x-1)]^{\frac{s}{2}}(2z)^m\prod_{k=1}^{m}(x-\cos \alpha_k^\ast)\right|\\
        & = 2^{\frac{s}{2}+m}\max_{x\in [-1,1]}\left|(x-1)^{\frac{s}{2}}\prod_{k=1}^{m}(x-\cos \alpha_k^\ast)\right|.
    \end{align*}
    However, the minimiser of the last expression is precisely $T_{m}^{(\frac{s}{2},0)}$ so we conclude that 
    \[
    T_{m}^{(\frac{s}{2},0)}(x) = \prod_{k=1}^{m}(x-\cos\alpha_k^\ast).
    \]
    This proves the even cases of \eqref{eq:jacobi} and \eqref{eq:jacobi_norms}.

    $\mathbf{n}$ \textbf{odd:} Write $n = 2m+1$. Due to symmetry of both $\T$ and the weight function $w_s$, we have that 
    \[
    \mathring{T}_n^{w_s}(z) = (z+\sigma)\prod_{k=1}^{2m}(z-e^{i\alpha_k^\ast})
    \]
    where $\alpha_k^\ast = 2\pi - \alpha_{k+m}^\ast$ and $\sigma \in \{-1,1\}$. We proceed with showing that $\sigma = 1$. Indeed, in order to arrive at a contradiction, assume to the contrary that $\sigma = -1$. Theorem \ref{thm:generalised_LSV} implies that the minimiser corresponding to \eqref{eq:fixed_zeros} is unique and therefore since we are currently assuming that $\sigma = -1$, we find that
    \[
    \mathring{T}_{2m+1}^{w_{s}}(z) = (z-1)\mathring{T}_{2m}^{w_{s+1}}(z).
    \]
    Upon multiplying both sides with $(z-1)^s$ and then differentiating, Theorem \ref{thm:generalised_LSV} implies that
    \[w_{s-1}(z)T_{2m+1}^{w_{s-1}}(z) = w_{s}(z)T_{2m}^{w_s}(z).\]
    However, it can easily be seen that $T_n^{w_s}$ has all its zeros in $\D$ for any combination of $s$ and $n$, see the beginning of the proof of \cite[Theorem 2.3]{lsv79}. As a consequence, $w_{s-1}T_{2m+1}^{w_{s-1}}$ has a zero at $z=-1$ of multiplicity $s-1$ while the zero of $w_{s}T_{2m}^{w_s}$ at $z =-1$ has multiplicity $s$. Since this is contradictory, we conclude that $\sigma = 1$ and gather that
\[
    \mathring{T}_n^{w_s}(z) = (z+1)\prod_{k=1}^{2m}(z-e^{i\alpha_k^\ast}).
    \]

Repeating the exact same procedure as in the even case gives us that 
    \begin{equation*}
       \|w_{s}\mathring{T}_{n}^{w_s}\|_\T = 2^{\frac{s+1}{2}+m}\max_{x\in [-1,1]}\left|(x-1)^{\frac{s}{2}}(x+1)^{\frac{1}{2}}\prod_{k=1}^{m}(x-\cos \alpha_k^\ast)\right|
    \end{equation*}
    and we conclude that 
    \[
    T_{m}^{(\frac{s}{2},\frac{1}{2})}(x) = \prod_{k=1}^{m}(x-\cos\alpha_k^\ast).
    \]
    This completes the proof.
\end{proof}

The combination of Theorems \ref{thm:generalised_LSV} and \ref{thm:jacobi} is remarkable since it allows seemingly different classes of Chebyshev polynomials to be related with each other. The class of Chebyshev polynomials corresponding to Jacobi weights on the interval $[-1,1]$ is somewhat more well-studied than weighted Chebyshev polynomials on the unit circle. As an implication of this, we show that this enables the computation of the norm asymptotics of $w_sT_n^{w_s}$ as described in Theorem \ref{thm:norm_limits}. First we recall the following result of Bernstein \cite{bernstein30-31} as it is stated in \cite{cer23}. It is a consequence of the alternation theorem.

 \begin{theorem*}[Bernstein \cite{bernstein30-31}]
			Suppose $\alpha_k\in \R$ and $b_k\in [-1,1]$ for $k=0, 1, \ldots, m$.  Consider a weight function $w:[-1,1]\rightarrow [0,\infty)$ of the form
			\begin{equation}  \label{weight}
				w(x) = w_0(x)\prod_{k=0}^{m}|x-b_k|^{\alpha_k},
			\end{equation}
			where $w_0$ is Riemann integrable and satisfies $1/M \leq w_0(x) \leq M$ for some constant $M\geq 1$. Then 
			\begin{equation}  \label{B asymp}
				\Vert w T_n^w \Vert_{[-1, 1]}=
				2^{1-n}\exp\left\{\frac{1}{\pi}\int_{-1}^{1}
				\frac{\log w(x)}{\sqrt{1-x^2}}dx\right\}
				\bigl(1+o(1)\bigr)
			\end{equation}
			as $n\rightarrow \infty$.
			\label{thm:Bernstein}
		\end{theorem*}

Using \cite[Lemma 2]{cer23}, it follows that
\[\frac{1}{\pi}\int_{-1}^{1}\frac{\log|x-b|}{\sqrt{1-x^2}}dx = -\log 2\]
for any $b\in [-1,1]$. We therefore find that
\begin{equation}
    \frac{1}{\pi}\int_{-1}^{1}\frac{\log |x-1|^\alpha|x+1|^\beta}{\sqrt{1-x^2}}dx = -(\alpha+\beta)\log 2.
    \label{eq:potential}
\end{equation}

\begin{proof}[Proof of Theorem \ref{thm:norm_limits}.]
We know from the combination of Theorems \ref{thm:generalised_LSV} and \ref{thm:jacobi} that
\[2\|w_sT_n^{w_s}\|_{\T} = \|w_{s+1}T_{n}^{w_{s+1}}\|_{\T} = \begin{cases}
        2^{\frac{s+1+n}{2}}\|w^{(\frac{s+1}{2},0)}T_{\frac{n}{2}}^{(\frac{s+1}{2},0)}\|_{[-1,1]}& n \text{ even,}\\
        2^{\frac{s+1+n}{2}}\|w^{(\frac{s+1}{2},\frac{1}{2})}T_{\frac{n-1}{2}}^{(\frac{s+1}{2},\frac{1}{2})}\|_{[-1,1]}& n \text{ odd.}
    \end{cases}\]
We study the even and odd case separately. Using Theorem \ref{thm:Bernstein} together with \eqref{eq:potential}, we conclude that
\begin{align*}
    2^{\frac{s+1+n}{2}}\|w^{(\frac{s+1}{2},0)}T_{\frac{n}{2}}^{(\frac{s+1}{2},0)}\|_{[-1,1]}&\sim 2^{\frac{s+1+n}{2}}2^{1-\frac{n}{2}}2^{-\frac{s+1}{2}} = 2, \\
    2^{\frac{s+1+n}{2}}\|w^{(\frac{s+1}{2},\frac{1}{2})}T_{\frac{n-1}{2}}^{(\frac{s+1}{2},\frac{1}{2})}\|_{[-1,1]}&\sim 2^{\frac{s+1+n}{2}}2^{1-\frac{n-1}{2}}2^{-\frac{s+1}{2}-\frac{1}{2}} = 2,
\end{align*}
where the limits as $n\rightarrow \infty$ are taken along even and odd integers separately. In either case we see that 
\[\|w_sT_n^{w_s}\|_{\T}\sim 1\] as $n\rightarrow \infty$.

To see that the sequence $\|w_sT_n^{w_s}\|_\T$ is monotonically decreasing we simply observe that
\[\|w_sT_n^{w_s}\|_\T = \|zw_sT_n^{w_s}\|_\T\geq \|w_sT_{n+1}^{w_s}\|_\T.\qedhere
\popQED\]\end{proof}
One immediate consequence of Theorem \ref{thm:norm_limits} is the fact that $\|w_sT_n^{w_s}\|_\T\geq 1$ for any $n$. As we remarked before, this can also be deduced by other means. For instance, it is a consequence of \cite[Theorem 12]{nsz21} but it can also easily be proven using the maximum principle. Alternatively, it is a consequence of Szeg\H{o}'s inequality \cite{szego24} if we take an appropriate power of the expression, see \cite[Theorem 5.5.4]{ransford95}. 

We can use the solution of Hal\'{a}sz's problem to get lower bounds for the convergence. Indeed, by \eqref{eq:halasz_sol} it follows that
\[\|w_1T_n^{w_1}\|_\T = \left[\cos\left(\frac{\pi}{2(n+2)}\right)\right]^{-(n+2)} = 1+\frac{\pi^2}{8n}+O(1/n^{2}).\]
One can show, for any choice of positive integers $p$ and $q$, that
\begin{equation}
    \|w_{p/q}T_n^{w_{p/q}}\|_{\T} \geq 1+\frac{\pi^2}{8nq^2}+O(1/n^2)
    \label{eq:convergence_speed}
\end{equation}
as $n\rightarrow \infty$.

It is illustrative to see how this can be deduced and so we take $w_{1/2}$ as an example. Since $(w_{1/2}T_n^{w_{1/2}})^{2}$ is a monic polynomial of degree $2n+1$ with a zero at $z = 1$, we find by comparison that
\[\|w_{1/2}T_n^{w_{1/2}}\|_{\T}^2\geq \|w_1T_{2n}^{w_1}\|_{\T} = \left[\cos\left(\frac{\pi}{2(2n+2)}\right)\right]^{-(2n+2)}.\]
By taking the square root of both sides, we conclude that
\[\|w_{1/2}T_n^{w_{1/2}}\|_{\T}\geq \left[\cos\left(\frac{\pi}{2(2n+2)}\right)\right]^{-(n+1)}= 1+\frac{\pi^2}{32 n}+O(1/n^2).\]

The final result that we wish to settle is Theorem \ref{thm:zero_limits}. This entails proving that the normalised zero counting measure of $T_n^{w_1}$ converges in the weak-star sense to the normalised Lebesgue measure on $\T$. This result can be interpreted using potential theory, since normalised Lebesgue measure on the unit circle coincides with equilibrium measure in that case, see \cite[Corollary 3.7.7]{ransford95}. For an arbitrary compact set $\sfE\subset \C$ there are quite general results which relate weak-star limits of zero counting measures corresponding to $T_n^{\sfE}$ with equilibrium measure on $\sfE$. Put somewhat vaguely, \cite[Theorem 2.3]{ms91} implies that any weak--star limit of the normalised zero counting measure corresponding to $T_n^{\sfE}$ will have a so-called ``balayage'' which coincides with equilibrium measure. To conclude that the weak-star limit is actually equal to the equilibrium measure, it is enough to conclude that there are at most $o(n)$ zeros of $T_n^{\sfE}$ in any compact set contained in the interior $\sfE$ as $n\rightarrow \infty$, see \cite[Theorem 2.1]{bss88}. Another approach to proving such a result, which is the key in our case to proving Theorem \ref{thm:zero_limits}, is to consider the limits of $|T_n^{\sfE}(z)|^{1/n}$ at points $z$ in the interior of $\sfE$. This technique has previously been used for instance in \cite[Corollary III.3.3]{minadiaz06} to determine weak-star limits of normalised zero counting measures corresponding to Bergman-Carleman polynomials and similarly in \cite[Theorem 2.5]{mdss05} in relation to weighted Bergman-Carleman polynomials. In the present situation where we are dealing with weighted Chebyshev polynomials on the unit circle corresponding to $w_1$, \cite[Theorem III.4.1]{st97} says that if $|w_1(z)T_n^{w_1}(z)|^{1/n}\rightarrow 1$ for some point $|z|<1$ as $n\rightarrow \infty$ then the normalised zero counting measure corresponding to $T_n^{w_1}$ will converge weak-star to the equilibrium measure on $\T$. This result is of fundamental importance to obtaining the result in our case since we can compute $T_n^{w_1}$ explicitly. This further highlights why the case $s=1$ is special. For other values of $s$ the polynomials appearing in Theorem \ref{thm:jacobi} do not possess known representations. 

\begin{proof}[Proof of Theorem \ref{thm:zero_limits}]
    Our strategy of proof is to show that $|w_1(z)T_n^{w_1}(z)|^{1/n}\rightarrow 1$ for some point $|z|<1$. For simplicity we choose $z = 0$ as our reference point and compute $T_n^{w_1}(0)$ explicitly. Hence if we can show that $|T_n^{w_1}(0)|^{1/n}\rightarrow 1$, then the result will follow. Theorems \ref{thm:generalised_LSV} and \ref{thm:jacobi} give us the equations
    \begin{align*}
        w_1(z)T_n^{w_1}(z) &= \frac{1}{2+n}\frac{d}{dz}\left(w_{2}(z)\mathring{T}_n^{w_2}(z)\right),
    \end{align*}
    and
     \begin{equation*}
        w_2(z)\mathring{T}_{n}^{w_{2}}(z) = \begin{cases}
            -(2z)^{\frac{2+n}{2}}w^{(1,0)}(x)T_{\frac{n}{2}}^{\left(1,0\right)}(x),& n\text{ even,} \\
            -(2z)^{\frac{2+n}{2}}w^{(1,\frac{1}{2})}(x)T_{\frac{n-1}{2}}^{\left(1,\frac{1}{2}\right)}(x) & n\text{ odd,}
        \end{cases}
    \end{equation*}
    where $x = (z+z^{-1})/2$. As before, we split the argument into the two separate cases and carry out all details since there are some subtle differences between the two cases.
    
    $\mathbf{n}$ \textbf{even:} Assume that $n = 2m$. Then we gather that
    \begin{equation}
        (z-1)^2\mathring{T}_{n}^{w_2}(z) = (2z)^{m+1}(x-1)T_m^{(1,0)}(x).
        \label{eq:s1case_even}
    \end{equation}
    We claim that $(x-1)T_m^{(1,0)}(x)$ coincides with a suitable scaling and translation of the classical Chebyshev polynomial of the first kind of degree $m+1$ which we denote by $T_{m+1}$. Recall that
    \begin{equation}
        T_k(x) = 2^{1-k}\cos(k\theta),\quad x = \cos \theta
        \label{eq:classical_cheb_first_kind}
    \end{equation}
    when $x\in [-1,1]$. The zeros of $\cos(k\theta)$ for $\theta\in [0,\pi]$ are precisely the values $\theta_j^{(k)} = \frac{(2j+1)\pi}{2k}$ for $j = 0,\dotsc,k-1$. Notice that the corresponding points on $[-1,1]$ are given by $\xi_j^{(k)} = \cos(\theta_j^{(k)})$, ordered from right to left. We introduce the monic polynomial $Q_{m+1}$ of degree $m+1$ defined by
    \[Q_{m+1}(x):=\left(\frac{2}{\xi_0^{(m+1)}+1}\right)^{m+1}T_{m+1}\left(\frac{x(\xi_0^{(m+1)}+1)+(\xi_0^{(m+1)}-1)}{2}\right).\]
    For any value of $m\geq 0$ this polynomial has a zero at $x = 1$. There are precisely $m+1$ points in $[-1,1]$ where $Q_{m+1}$ attains its maximal modulus. Furthermore, adjacent extremal points have alternating signs. Consequently, the alternation property in $\eqref{eq:alternation_theorem}$ is satisfied and so it follows from \cite[Theorem 7]{nsz21} that
    \begin{equation}
        (x-1)T_m^{(1,0)}(x) =Q_{m+1}(x).
        \label{eq:classical_representation}
    \end{equation}
    Writing $(x-1)T_m^{(1,0)}(x) = \sum_{k=0}^{m+1}a_kx^k$ with $a_{m+1} = 1$, we find from \eqref{eq:s1case_even} that
    \begin{align*}
        \frac{d}{dz}&\left((z-1)^2\mathring{T}_n^{w_2}(z)\right)\\& = 2^{m+1}\left\{(m+1)\left(\frac{z^{-1}}{2^{m+1}}+\frac{a_m}{2^{m}}\right)-\frac{1}{2}\left((m+1)\frac{z^{-1}}{2^m}+\frac{ma_m}{2^{m-1}}\right)\right\}+o(1)
    \end{align*}
    as $z\rightarrow 0$. Therefore,
    \[-T_{2m}^{w_1}(0) = w_1(0)T_{2m}^{w_1}(0)= \frac{a_m}{m+1}.\]
    We are left with determining $a_m$ which is the coefficient in front of $x^{m}$ of $Q_{m+1}$. Since $T_{m+1}$ is an even/odd polynomial when $m+1$ is an even/odd number, we find that
    \[a_m = (m+1)\cdot\frac{\xi_0^{(m+1)}-1}{\xi_0^{(m+1)}+1}.\]
    Hence
    \[-T_{2m}^{w_1}(0) = \frac{\xi_0^{(m+1)}-1}{\xi_0^{(m+1)}+1} = \frac{\pi^2}{4(m+1)^2}+O(m^{-3})\]
    from which we see that
    \[|T_{2m}^{w_1}(0)|^{1/2m}\rightarrow 1.\]
    This completes the study of the even case.

    $\mathbf{n}$ \textbf{odd:} Assume that $n = 2m+1$. Then
    \begin{align}
        \begin{split}
            (z-1)^2\mathring{T}_{2m+1}^{w_2}(z) & = (2z)^{m+1+\frac{1}{2}}(1+x)^{1/2}(x-1)T_{m}^{(1,\frac{1}{2})}(x) \\
            & = 2^{m+1}z^{m+1}(z+1)(x-1)T_{m}^{(1,\frac{1}{2})}(x).
        \end{split}
        \label{eq:cheb_jacobi_odd}
    \end{align}
    We write $(x-1)T_m^{(1,\frac{1}{2})}(x) = \sum_{k=0}^{m+1}b_kx^k$ with $b_{m+1} = 1$. Differentiating \eqref{eq:cheb_jacobi_odd}, we find that
    \begin{align*}
        \frac{d}{dz}\left\{(z-1)^2\mathring{T}_{2m+1}^{w_2}(z)\right\}  = 1+2b_m+O(z)
    \end{align*}
    as $z\rightarrow 0$. Therefore,
    \[-T_{2m+1}^{w_1}(0) = w_1(0)T_{2m+1}^{w_1}(0) = \frac{1+2b_m}{2m+3}.\]
    The final ingredient to analysing the behaviour of $T_{2m+1}^{w_1}$ is to determine $b_m$.

    With the change of variables $x = \cos \theta$, together with elementary trigonometry, we see that
    \[(1+x)^{1/2} = \sqrt{2}\cos\frac{\theta}{2}.\]
    Let $V_{k}$ denote the Chebyshev polynomial of the third kind of degree $k$, normalised to have leading coefficient $1$. That is,
    \[V_{k}(x) = 2^{-k}\frac{\cos((k+\frac{1}{2})\theta)}{\cos\frac{\theta}{2}},\quad x = \cos \theta.\]
    $V_k$ is the monic polynomial of degree $k$ which deviates the least from zero on $[-1,1]$ with respect to multiplication by the weight function $\sqrt{1+x}$, see e.g. \cite[Property 1.1]{mason93}. The zeros of $\cos (k+\frac{1}{2})\theta$ are given by $\vartheta_j^{(k)} = \frac{\pi}{2k+1}$ for $j = 0,\dotsc,k-1$. The corresponding points on $[-1,1]$ are thus given by $\eta_j^{(k)} = \cos \vartheta_j^{(k)}$, again ordered from right to left. Since the monic polynomial
    \begin{align}
        \begin{split}
            R_{m+1}(x):=\left(\frac{2}{\eta_0^{(m+1)}+1}\right)^{m+1}V_{m+1}\left(\frac{x(\eta_0^{(m+1)}+1)+(\eta_0^{(m+1)}-1)}{2}\right).
        \end{split}
        \label{eq:third_kind}
    \end{align}
    has a zero at $1$ and satisfies the alternation property from \eqref{eq:alternation_theorem} for $m+1$ points on $[-1,1]$ we conclude from \cite[Theorem 7]{nsz21} that $R_{m+1}$ is actually equal to $(x-1)T_{m}^{(1,\frac{1}{2})}(x)$. The Chebyshev polynomials of the third kind can be represented using Chebyshev polynomials of the first kind together with a change of variables. Explicitly we have that 
    \[V_k(x) = 2^{k}\sqrt{\frac{2}{1+x}}T_{2k+1}\left(\sqrt{\frac{x+1}{2}}\right)\]
    where $T_{2k+1}$ is the Chebyshev polynomial of the first kind as in \eqref{eq:classical_cheb_first_kind}, see \cite{mason93}. We can therefore use explicit formulas for the coefficients of $T_{2k+1}$ to get formulas for the coefficients $V_{k}$, see e.g. \cite{rivlin74}. This gives us that
    \[b_m = (m+1)-\frac{2m+3}{\eta_0^{(m+1)}+1}.\]
    Since $\eta_0^{(m+1)} = 1-\frac{\pi^2}{2(2m+3)^2}+O(m^{-4})$, it follows that
    \begin{align*}
        b_m  = -\frac{1}{2}-\frac{\pi^2}{8(2m+3)}+O(m^{-2})
    \end{align*}
    as $m\rightarrow \infty$ and therefore
    \[T_{2m+1}^{w_1}(0) = \frac{\pi^2}{4(2m+3)^2}+O(m^{-3}).\]
    From this we finally conclude that
    \[|T_{2m+1}^{w_1}(0)|^{\frac{1}{2m+1}}\rightarrow 1\]
    as $m\rightarrow \infty$. 

    In conclusion, we see that $|T_n^{w_1}(0)|^{\frac{1}{n}}\rightarrow 1$ and therefore \cite[Theorem III.4.1]{st97} implies that $\nu_{n,1}$ converges in the weak-star sense to the equilibrium measure on $\T$.
\end{proof}
The method of proof relied on the explicit representation of the polynomials $(x-1)T_{m}^{(1,\frac{1}{2})}$ and $(x-1)T_m^{(1,0)}$ which is available to us in the form of classical Chebyshev polynomials. For general $T_n^{(\alpha,\beta)}$ an explicit representation seems to be unknown. One difficulty of obtaining asymptotic formulae in this case, lies in the fact that the weights $(1-x)^\alpha(1+x)^{\beta}$ vanish at the end points of $[-1,1]$. For non-vanishing weights, several results concerning the asymptotical behaviour of the corresponding weighted Chebyshev polynomials exist, see e.g. \cite{kp08, kroo14}.
\end{section}

\begin{section}{Concluding remarks and numerical examples}
\label{sec:conclusion}    
The contribution in this article consists of extending two previous results beyond the regime of polynomials to fractional powers of polynomials. We showed that the results in \cite{lsv79} which deals with weighted Chebyshev polynomials on the unit circle can be lifted to include weights with non-integer powers. As an intermediate step in proving this, we established an extension of the Erd\H{o}s--Lax inequality. We will discuss the outlook of these results separately. 

We begin, however, with a brief discussion concerning the Chebyshev polynomials corresponding to the sets $\sfE_m$ defined in \eqref{eq:lemniscate}. From \eqref{eq:lemniscate_cheb}, we conclude the relation
\[\|T_{nm+l}^{\sfE_m}\|_{\sfE_m} = \|w_{l/m}T_{n}^{w_{l/m}}\|_{\T}.\]
As a corollary to Theorem \ref{thm:norm_limits}, this characterisation implies the following result.
\begin{corollary}
    With $\sfE_m$ defined as in \eqref{eq:lemniscate} and $T_{n}^{\sfE_m}$ denoting the corresponding Chebyshev polynomials, it holds that
\[\lim_{n\rightarrow \infty}\|T_n^{\sfE_m}\|_{\sfE_m} = 1\]
for any $m\in \N$.
\end{corollary}
At the same time \eqref{eq:convergence_speed} limits the decay rate to $\|T_n^{\sfE_m}\|_{\sfE_m}\geq 1+\frac{\pi^2}{8nm^2}+O(n^{-2})$ when $n$ is not divisible by $m$. This should be compared with \eqref{eq:Faber_convergence} where the norms of Chebyshev polynomials corresponding to the closure of an analytic Jordan domains are shown to decay to $1$ exponentially fast. The external conformal map corresponding to $\sfE_m$ is for $|z|>1$ given by $\Phi(z) = (z^m+1)^{1/m}$ where the branch is chosen such that $\Phi(z) = z+O(1)$ as $z\rightarrow \infty$. While $\Phi$ can be conformally extended across the boundary of $\T$ away from the $m$th roots of $-1$, the set $\sfE_m$ is not the closure of a Jordan domain. Hence, \eqref{eq:Faber_convergence} can not be expected to hold.

Theorems \ref{thm:generalised_LSV} and \ref{thm:jacobi} further imply that the classes $\{T_{n}^{\sfE_m}\}$, $\{T_{n}^{w_s}\}$ and $\{T_{n}^{(\alpha,\beta)}\}$ are all related. Determining one of these classes of polynomials would shed light on the remaining ones.

Our first suspicions of the validity of Theorem \ref{thm:generalised_LSV} came from numerical computations using a generalisation of the Remez algorithm due to Tang, see \cite{remez34-1, remez34-2, tang88}. Using this algorithm, it is possible to compute $w_sT_n^{w_s}$ for any $s>0$. By further considering the primitive of $w_sT_n^{w_s}$ conditioned to have a zero of order $s+1$ at $z=1$ one sees that all the zeros of this primitive function are situated on the unit circle. Numerical computations further hinted at a property of the zeros of $w_sT_n^{w_s}$, namely that they asymptotically approach the boundary $\T$, see Figure \ref{fig:zeros}. 

\begin{figure}[h!]
\begin{subfigure}{.33\textwidth}
  \centering
  \includegraphics[width=\linewidth]{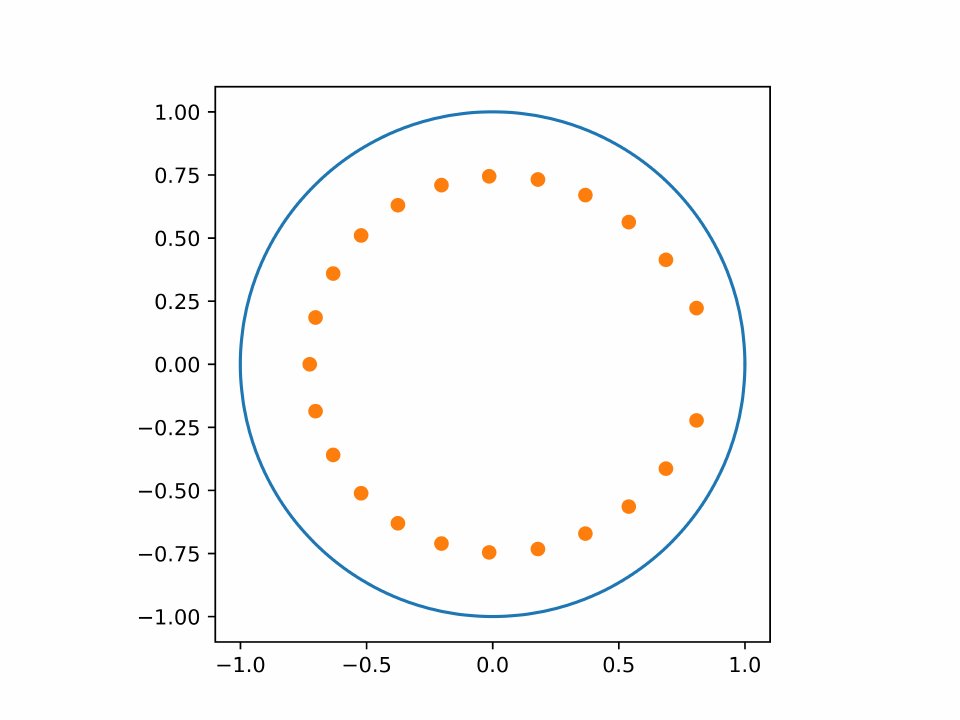}
  \caption{$T_{23}^{w_{1/2}}$}
  \label{fig:sfig1}
\end{subfigure}%
\begin{subfigure}{.33\textwidth}
  \centering
  \includegraphics[width=\linewidth]{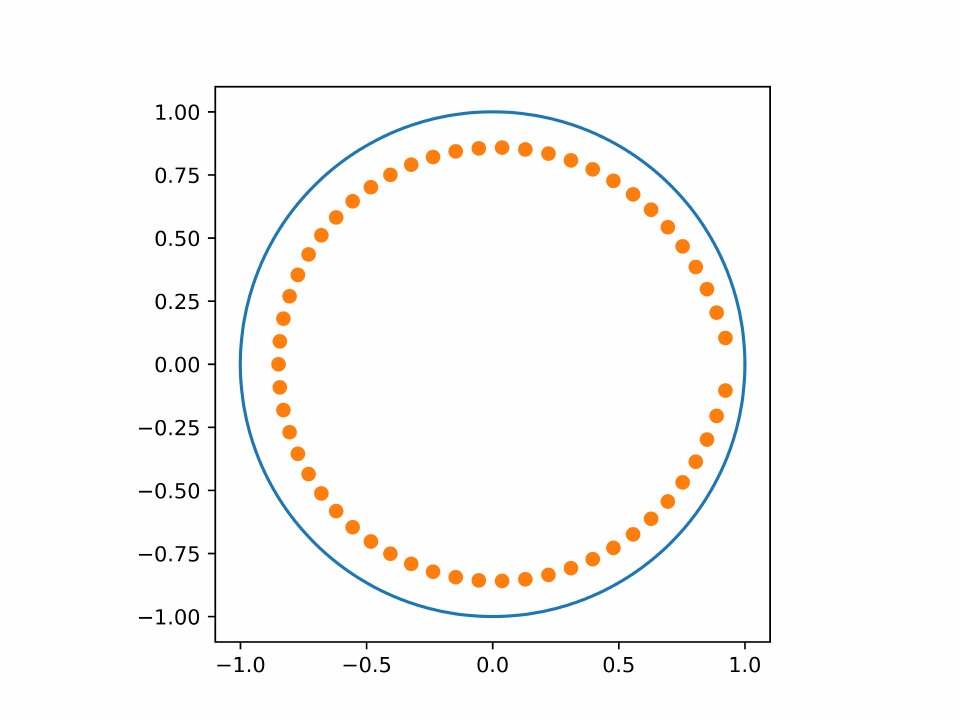}
  \caption{$T_{57}^{w_{1/2}}$}
  \label{fig:sfig2}
\end{subfigure}
\begin{subfigure}{.33\textwidth}
  \centering
   \includegraphics[width=\linewidth]{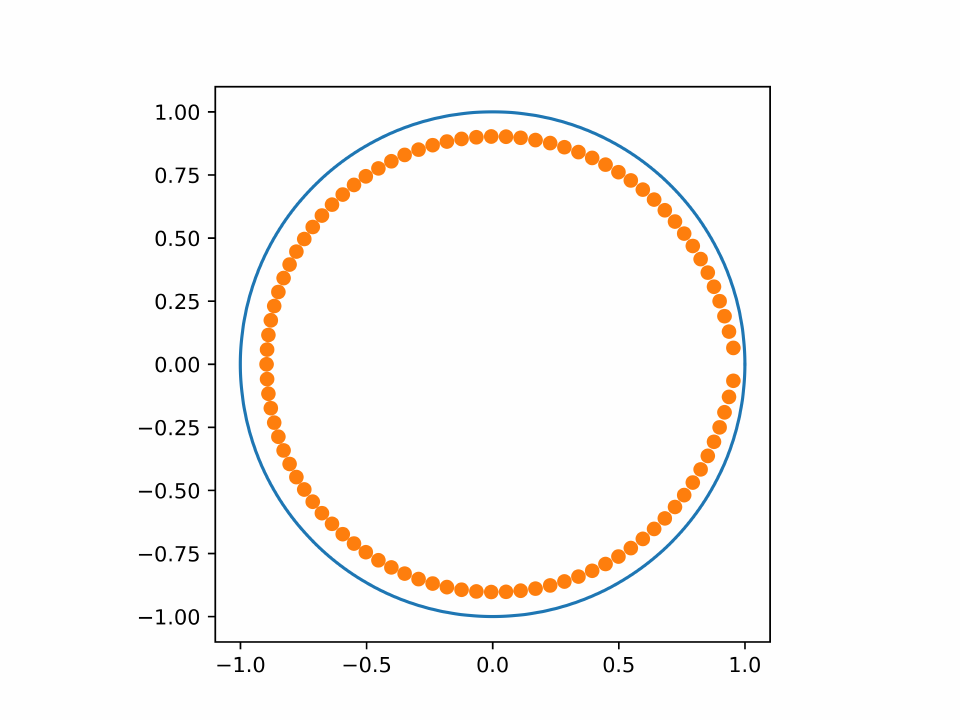}
  \caption{$T_{95}^{w_{1/2}}$}
  \label{fig:sfig3}
\end{subfigure}
    \caption{The ``dots'' represent zeros of $T_n^{w_{1/2}}$.}
    \label{fig:zeros}
\end{figure}

Using the notation $\nu_{n,s}$ to denote the normalised zero counting measure as defined in \eqref{eq:counting_measure}, Theorem \ref{thm:zero_limits} implies that $\nu_{n,1}$ in fact converges to the measure $\frac{d\theta}{2\pi}$ on $\T$. Based on the plots in Figure \ref{fig:zeros}, and similar ones for other values of $s$, we hypothesise that the result extends for all $s>0$.

\begin{conjecture}
    For any value of $s>0$, the measure $\nu_{n,s}$ converges in the weak-star sense to the measure $\frac{d\theta}{2\pi}$ on $\T$.
    \label{conjecture:zeros}
\end{conjecture}
From \cite[Theorem 1.7]{ab01} we gather that it would be enough to settle that \[\nu_{n,s}\Big(\{z: |z|\leq r\}\Big)\rightarrow 0,\quad n\rightarrow \infty\] for all fixed $r\in (0,1)$. This behaviour is strongly suggested by numerical simulations. The reason that the method used in the proof of Theorem \ref{thm:zero_limits} does not seem to carry over in this case is that the polynomials $T_{n}^{(\alpha,\beta)}$ do not possess a known explicit representation for arbitrary $\alpha$ and $\beta$. Due to the relation expressed in \eqref{eq:lemniscate_cheb}, the validity of this conjecture would further imply that the normalised zero counting measure corresponding to $T_{nm+l}^{\sfE_m}$ would converge to equilibrium measure on $\sfE_m$ as $n\rightarrow \infty$ for $l\neq 0$. In \cite{st90} the authors specifically mention the zero distribution for Chebyshev polynomials of $\sfE_2$ since the odd ones always have a zero at the origin which lies on the boundary, despite the boundary being an analytic curve. However, in this case the curve has a self intersection.

It is an interesting matter to consider zero distributions arising for sequences of monic polynomials exhibiting asymptotic minimisation within the unit disk. With asymptotic minimisation we mean sequences $P_n$ of monic polynomials of degree $n$ satisfying $\|P_n\|_{\T}\rightarrow 1$. Recall that any monic polynomial $P$ satisfies $\|P\|_\T\geq 1$ with equality if and only if $P$ is a monomial.

Consider fixing a zero located at $0\leq r<1$. The polynomial $P_n(z) = z^n-r^n$ has the properties $P_n(r) = 0$ and $\|P_n\|_{\T} = 1+r^n\rightarrow 1$. Intriguingly, despite the convergence of the norm, all zeros of $P_n$ are confined to the circle $\{z: |z| = r\}$. The polynomial $P_n$ is not minimal among all monic polynomials with a prescribed zero at $z = r$. However, numerical simulations indicate that the zeros of Chebyshev polynomials corresponding to weights of the form $|z-r|$ with $0<r<1$ distribute on the circle of radius $r$ in a similar manner.

These considerations in particular illuminate possible weak-star limits of the zero distributions corresponding to asymptotically minimising sequences. For fixed $0\leq r_k<1$, $k = 1,\dotsc,m$ the polynomials
\[P(z,n_0,\dotsc,n_m) = w_1(z)T_{n_0}^{w_1}(z)\prod_{k=1}^{m}(z^{n_k}-r_k^{n_k})\]
will be asymptotically minimising in the sense that $\|P(\cdot, n_0,\dotsc,n_m)\|_\T\rightarrow 1$ if all $n_k\rightarrow \infty$. The corresponding weak-star limit of the normalised zero distribution
\[\nu_{n_0,\dotsc,n_m} = \frac{1}{\sum_{k=0}^{m} n_k}\sum_{\{ z\, :\, P(z,n_0,\dotsc,n_m) = 0\}}\delta_z\]
will be a convex combination of normalised arc-length measures on the circles $\{z: |z| = r_k\}$ and $\T$. This points out one of the main difficulties in settling Conjecture \ref{conjecture:zeros}; it is easy to produce asymptotically minimising sequences of polynomials whose normalised zero distributions have very different behaviours.

Something that seem to further complicate matters concerns the speed with which the zeros of $T_n^{w_s}$ approach the boundary. Based on numerical experiments such as the one illustrated in Figure \ref{fig:zeros} it does not seem like the zeros approach the boundary particularly fast. An intuitive approach to motivate why this the case is that if $P$ is any monic polynomial with all its zeros on $\T$ then the combination of Theorems \ref{thm:generalised_LSV} and \ref{thm:norm_limits} imply that
\[\|w_sP\|_\T\geq 2, \quad s\geq 1.\]
This shows that the zeros of an asymptotically minimising sequence, with norm converging to $1$, should not lie on the boundary. Another intuitive argument concerning the case for $s = 1$ comes from considering the trial polynomials
\[ Q_n(z) = \frac{z^n-r_n^n}{z-r_n}(z-1)\]
where $r_n\in (0,1)$ for each $n$. These polynomials are monic and have a fixed zero at $z = 1$ with all other zeros lying on the circle of radius $r_n$. We consider the case where $n$ is odd since then it is easily seen that \[\|Q_n\|_\T  = |Q_n(-1)| = 2\cdot \frac{1+r_n^n}{1+r_n}.\] If the zeros of $Q_n$ are to approach $\T$ then it is necessary that $r_n\rightarrow 1$ as $n\rightarrow\infty$. On the other hand we must have that $r_n^n\rightarrow 0$ as $n\rightarrow \infty$ if we want the norms to be asymptotically minimising, i.e. $\|Q_n\|_\T\rightarrow 1$. This shows that the sequence of trial minimisers $\{Q_n\}$ will be asymptotically minimising if and only if $r_n\rightarrow 1$ and $r_n^{n}\rightarrow 0$. For instance we can put $r_n = 1-n^{-\alpha}$ for any $\alpha\in (0,1)$. The fact that $r_n^n\rightarrow 0$ shows that the speed with which the zeros approach the boundary is upper bounded if we are to have that this class of trial polynomials are asymptotically minimising.

We finally turn to consider the Erd\H{o}s--Lax type results proven in Section \ref{section:el}. Theorem \ref{thm:erdos_lax_circle} is a direct generalisation of a result due to P\'{o}lya and Szeg\H{o} that first appeared in print in \cite{lax44}. Theorem \ref{thm:generalised_erdos_lax} generalises the Erd\H{o}s--Lax inequality. We believe that a similar result should be valid without putting restrictions on the zeros in the exterior of the closed unit disc. We conjecture the following.

\begin{conjecture}
    Let $P(z) = c\prod_{k=1}^{N}(z-a_k)^{ms_k}$ where $c\in \mathbb{C}$, $m\in \mathbb{N}$, $s_k\geq 1$, $ms_k\in \N$, and $|a_k|\geq 1$. Then for any branch $f(z) = P(z)^{1/m}$, we have
    \[\|f'\|_{\T}\leq \frac{\sum_{k=1}^{N} s_k}{2}\|f\|_{\T}\]
    with equality if and only if $|a_k| = 1$ for all $k$.
\end{conjecture}
The reason that the techniques developed in this article do not seem to extend to prove this case is that for such functions $f$, it is not guaranteed that $f+\zeta f^\ast$ from the proof of Theorem \ref{thm:generalised_erdos_lax} is a rational power of a polynomial. Instead, we believe that new ideas are needed to prove this conjecture.

\end{section}
\section*{Acknowledgement} We express our gratitude towards Associate Professor Jacob Stordal Christiansen for helpful discussions when preparing this manuscript. 

\end{document}